\documentclass[12pt,a4paper,reqno]{amsart}

\usepackage[utf8]{inputenc}
\usepackage[english]{babel}
\usepackage{amsthm}
\usepackage{amssymb}
\usepackage[abbrev]{amsrefs}
\usepackage{mathtools}
\usepackage{cancel}
\usepackage[dvipsnames]{xcolor}
\usepackage{bm}
\usepackage[shortlabels]{enumitem}
\usepackage{hyperref}
\usepackage{bbm}
\usepackage{MnSymbol}

\newtheorem{thm}{}[section]
\newtheorem{theorem}[thm]{Theorem}

\newtheorem{lemma}[thm]{Lemma}

\theoremstyle{definition}
\newtheorem{definition}[thm]{Definition}
\theoremstyle{remark}
\newtheorem{remark}[thm]{Remark}

\numberwithin{equation}{section}
\allowdisplaybreaks

\newcommand{\FF}{\ensuremath{\mathbb{F}}}

\newcommand{\RR}{\ensuremath{\mathbb{R}}}

\newcommand{\CC}{\ensuremath{\mathbb{C}}}
\newcommand{\C}{\ensuremath{\mathbf{C}}}
\newcommand{\K}{\ensuremath{\mathbf{K}}}
\newcommand{\D}{\ensuremath{\mathbf{D}}}
\newcommand{\NN}{\ensuremath{\mathbb{N}}}
\newcommand{\xx}{\ensuremath{\bm{x}}}

\newcommand{\yy}{\ensuremath{\bm{y}}}

\newcommand{\XX}{\ensuremath{\mathbb{X}}}
\newcommand{\XB}{\ensuremath{\mathcal{X}}}
\newcommand{\YB}{\ensuremath{\mathcal{Y}}}

\newcommand{\BB}{\ensuremath{\mathcal{B}}}

\newcommand{\Ind}{\ensuremath{\mathbbm{1}}}

\newcommand{\EE}{\ensuremath{\mathbb{E}}}

\DeclareMathOperator{\sgn}{sign}
\DeclareMathOperator{\spn}{span}
\DeclareMathOperator{\supp}{supp}

\AtBeginDocument{\def\MR#1{}}

\usepackage[left=2.5cm,right=2.5cm,top=4cm,bottom=4cm]{geometry}

\title[Polynomials and greedy bases]{Approximation by polynomials with constant coefficients and the Thresholding Greedy Algorithm}

\author[M. Berasategui]{Miguel Berasategui}

\address{Miguel Berasategui
	\\
	IMAS - UBA - CONICET - Pab I, Facultad de Ciencias Exactas y Naturales \\ Universidad de Buenos Aires \\ (1428), Buenos Aires, Argentina}
\email{mberasategui@dm.uba.ar}

\author[P.M. Bern\'a]{Pablo M. Bern\'a}
\address{Pablo M. Bern\'a\\
	Departamento de Métodos Cuantitativos, CUNEF Universidad\\ 
	Madrid\\
	28040 Spain} 
\email{pablo.berna@cunef.edu}

\author[D. Gonz\'alez]{David Gonz\'alez}
\address{David Gonz\'alez\\
	Universidad San Pablo CEU, CEU Universities \\ 
	Madrid\\
	28003, Spain}
\email{david.gonzalezmoro@usp.ceu.es}

\begin{document}
	
	\begin{abstract}
		Greedy bases are those bases where the Thresholding Greedy Algorithm (introduced by S. V. Konyagin and V. N. Temlyakov) produces the best possible approximation up to a constant. In 2017, Berná and Blasco gave a characterization of these bases using polynomials with constant coefficients. In this paper, we continue this study improving some optimization problems and extending some results to the context of quasi-Banach spaces.
		
	\end{abstract}
	\maketitle
	\section{Introduction}
	Since 1999, one of the most important algorithms in the field of non-linear approximation theory in quasi-Banach spaces is the Thresholding Greedy Algorithm $(G_m)_{m\in\mathbb N}$, where, for a given element $f$ in the space $\mathbb X$,  the algorithm selects the largest coefficients of $f$ (in modulus) respect to a given basis $(\xx_j)_{j=1}^\infty$ in the space. This algorithm was introduced by S. V. Konyagin and V. N. Temlyakov in \cite{KT} and the convergence of the TGA has been extensively studied from different perspectives by various researchers, such as F. Albiac, J. L. Ansorena, D. Kutzarova, N. J. Kalton, S. J. Dilworth and P. Wojtaszczyk (\cite{AAW2019, DKK2003,DKKT2003,W2000}), among others. Thanks to these studies, in the literature we can find the so called \textit{greedy-like bases}, that is, bases where the TGA converges in some way.

	One of the most important contributions to the theory of greedy-like bases was introduced in \cite{KT}, where the authors established the notion of a greedy basis as the one for which the TGA produces the best possible approximation up to a constant, that is, for every $f\in\mathbb N$ and $m\in\mathbb N$,
	\begin{eqnarray}\label{g1}
		\Vert f-G_m(f)\Vert \approx\inf \left\lbrace \left\Vert f-\sum_{n\in A}a_n\xx_n  \right\Vert : a_n's\in\mathbb C, \vert A\vert\leq m\right\rbrace.
	\end{eqnarray}
	In \cite{BB2017}, P. M. Berná and Ó. Blasco proved an equivalent version of greediness showing that a basis is greedy if and only if for every $f\in\mathbb N$ and $m\in\mathbb N$,
	\begin{eqnarray}\label{g2}
		\Vert f-G_m(f)\Vert \lesssim \inf \left\lbrace \left\Vert f-\sum_{j\in A}\alpha \xx_j\right\Vert : \alpha\in\mathbb C, \vert A\vert\leq m\right\rbrace,
	\end{eqnarray}
	The element $y=\sum_{j\in A}\alpha \xx_j$ is called a \textit{polynomial with constant coefficients}. 
	
	In this paper, we further advance the theory initiated by P. M. Berná and Ó. Blasco. Specifically, we extend the theory of \cite{BB2017} to the context of quasi-Banach spaces, examining whether the characterizations of greedy-like bases using polynomials of constant coefficients are applicable in this setting. Furthermore, we aim to relax the condition \eqref{g2} even further by proving that a basis is greedy if and only if \eqref{g2} is satisfied for a specific value of $\alpha$.

The structure of the paper is as follows: in Section \ref{notation}, we introduce the main definitions that we use and the main results that we prove. In Section \ref{tech}, we prove some techincal results concerning some properties about symmetry and unconditionality, where respect to the last one, we provide a new characterization in terms of elements with possitive coefficients. In Sections \ref{proof1} and \ref{proof2}, we prove Theorems~\ref{th1} and ~\ref{thag} where we characterize greediness and almost-greediness in terms of polynomials of constant coefficients with $\alpha$ fixed.

	\section{Notation and main definitions}\label{notation}
	A \textit{quasi-norm} on a vector space $\mathbb X$ over $\mathbb F=\mathbb R$ or $\mathbb C$ is a map $\Vert\cdot\Vert:\XX\rightarrow [0,+\infty)$ verifying the following three conditions:
	\begin{itemize}
		\item[N1)] $\Vert f\Vert\geq 0$ for all $f\neq 0$ and $\Vert f\Vert =0$ if and only if $f=0$,
		\item[N2)] $\Vert tf\Vert=\vert t\vert\Vert f\Vert$ for all $t\in\mathbb F$ and all $f\in\XX$,
		\item[N3)] there is $c>0$ such that for all $f,g\in\XX$, $\Vert f+g\Vert\leq c(\Vert f\Vert+\Vert g\Vert)$.
	\end{itemize}
	Then, it is well known that the collection 
	$$\left\lbrace f\in\XX : \Vert f\Vert<\frac{1}{n}\right\rbrace,\, n\in\mathbb N,$$
	is a base of neighbourhoods of zero, so the quasi-norm induces a metrizable linear topology and if $\XX$ is complete, we say that $(\XX,\Vert \cdot\Vert)$ is a \textit{quasi-Banach space}. 
	
	Given $0<p\leq 1$, we remind that a map $\Vert\cdot\Vert:\XX\rightarrow [0,+\infty)$ is a \textit{$p$-norm} if the map verifies the conditions N1), N2) and 
	
	\begin{itemize}
		\item[N4)] for all $f,g\in\XX$, $\Vert f+g\Vert^p\leq \Vert f\Vert^p+\Vert g\Vert^p$.
	\end{itemize}
	Of course, N4) implies N3) with $c=2^{1/p-1}$ and every quasi-Banach space whose quasi-norm is a $p$-norm is called a \textit{$p$-Banach space} and thanks to the Aoiki-Rolewicz's Theory (see \cite{Aoki,Rolewicz}), every quasi-Banach space becomes $p$-Banach under a suitable renorming.
	
	In the case of $p$-Banach spaces, we define the following geometrical constants
	\begin{eqnarray}\label{ap}
		A_p:= (2^p-1)^{-1/p}
	\end{eqnarray}
	and
	\begin{eqnarray}\label{bp}
		B_p:=\begin{cases} 2^{1/p}  A_p& \mbox{ if }\FF=\RR,\\  4^{1/p}  A_p& \mbox{ if }\FF=\CC.\end{cases}
	\end{eqnarray}
	
The following is a useful result from \cite{AABW2021}  that allows us to circumvent in some context the lack of convexity when $0<p<1$. 
	
	\begin{lemma}\cite[Corollaries 2.3, 2.4]{AABW2021}\label{lemmapbanach} Let $\XX$ be a $p$-Banach space for some $0<p\leq 1$. Let $(a_j)_{j\in A}$ be any collection of vectors in $\mathbb X$ with $A\subset \NN$ a finite set. Then:
		\begin{align*}
			\left\Vert\sum_{j\in A} a_jf_j\right\Vert\le&B_p\max_{j\in A }\left\vert a_j\right\vert \max_{B\subset A}\left\Vert\sum_{j\in B} f_j\right\Vert;\\
			\left\Vert\sum_{j\in A} a_jf_j\right\Vert\le&A_p\max_{j\in A }\left\vert a_j\right\vert \sup\left\lbrace \left\Vert\sum_{j\in A} \varepsilon_j f_j\right\Vert : \vert\varepsilon_j\vert=1\right\rbrace.
		\end{align*}
	\end{lemma}

	Now, we consider that we have a \textit{basis} $\XB=(\xx_n)_{n\in\mathbb N}$ in a quasi-Banach space $\XX$ over the field $\FF=\RR$ or $\FF=\CC$, that is, $\XB$ is a collection of elements of the space satisfying the following conditions:
	
	\begin{itemize}
		\item $[\xx_n : n\in\mathbb N]=\XX$, where $[E]$ denotes the closure of the linear span generated by the elements of the subspace $E\subset\mathbb X$;
		\item there is a unique sequence of biorthogonal functions $\XB^*=(\xx_n^*)_{n\in\mathbb N}$ such that $\xx_n^*(\xx_m)=\delta_{n,m}$;
		\item there are are $c_1,c_2>0$ such that
		$$\sup_{n\in\mathbb N}\Vert \xx_n\Vert\leq c_1<\infty,\; \sup_{n\in\mathbb N}\Vert \xx_n^*\Vert_*\leq c_2<\infty. $$ 
\end{itemize}
If additionally $\XB^*$ has the property that $\xx_n^*(f)=0$ for all $n\in\mathbb N$ entails that $f=0$, we say that $\XB^*$ is \emph{total} or, equivalently, that $\XB$ is a \textit{Markushevich basis}. In the sequel, unless specified we will not assume that $\XB$ is a Markushevich basis  - though our results will impliy this property in all cases (see \cite[Corollary 4.5]{AABW2021}).

		Given a basis $\XB$, an element $f\in\XX$ and $m\in\mathbb N$, we define a \textit{greedy ordering} as a map $\pi: \mathbb N\rightarrow \mathbb N$ such that $\text{supp}(f)\subset\pi(\mathbb N)$ (where the support of $f$ is the collection of $n\in \mathbb N$ where $\xx_n^*(f)\neq 0$) and such that if $j<k$, then either $\vert \xx_{\pi(j)}^*(f)\vert> \vert \xx_{\pi(k)}^*(f)\vert$ or $\vert \xx_{\pi(j)}^*(f)\vert=\vert \xx_{\pi(k)}^*(f)\vert$ and $\pi(j)<\pi(k)$. Using this map, we define a \textit{greedy sum of $f$ of order $m$} like the sum
	$$G_m[\mathcal X,\mathbb X](f)=G_m(f):=\sum_{i=1}^m \xx_{\pi(i)}^*(f)\xx_{\pi(i)}.$$
	The Thresholding Greedy Algorithm (TGA) is the collection $(G_m)_{m\in\mathbb N}$. \\
Recall that the projection operator is defined as follows: given a finite set $A\subset\mathbb N$ and $f\in\mathbb X$, 
	$$P_A[\mathcal X,\mathbb X](f)=P_A(f):=\sum_{n\in A}\xx_n^*(f)\xx_n.$$
	Then,  a greedy sum is a projection over a \textit{greedy set} $A_m(f)$, where $A_m(f)$ is a set verifying the following two conditions: the cardinality is exactly $m$ and
	$$\min_{n\in A_m(f)}\vert\xx_n^*(f)\vert\geq\max_{n\not\in A_m(f)}\vert \xx_n^*(f)\vert.$$
	
	Of course, unconditionality guarantees the convergence of the TGA, where we say that a basis is $\K$-\textit{unconditional} with $\K>0$ if
	$$\Vert P_A(f)\Vert\leq \K\Vert f\Vert,\; \forall \vert A\vert<\infty, \forall f\in\mathbb X.$$
	However, in the literature, there are examples of conditional bases where the TGA converges. For that reason, we have the following definition to study the convergence of the algorithm. 
	\begin{definition}
		We say that a basis $\XB$ in a quasi-Banach space $\XX$ is \textit{quasi-greedy} if there is $\C>0$ such that for all $f\in\mathbb X$, 
\begin{eqnarray}\label{defqg}
	\Vert f-P_A(f)\Vert\leq \C\Vert f\Vert,
\end{eqnarray}
		whenever $A$ is a finite greedy set of $f$. The smallest constant verifying \eqref{defqg} is denoted by $\C_q$ and we say that $\mathcal X$ is $\C_q$-quasi-greedy.
	\end{definition}
	P. Wojtaszczyk proved in \cite{W2000} (also the proof could be found in \cite{AABW2021}) that a basis is quasi-greedy if and only if the TGA converges, that is,
	$$\lim_{m\rightarrow+\infty}\left\Vert f-\sum_{n=1}^m \xx_{\pi(n)}^*(f)\xx_{\pi(n)}\right\Vert=0.$$
	
	An stronger notion than quasi-greediness is greediness, that is, the case when the TGA produces the best possible approximation up to a constant. To define it precisely, we need the \textit{best $m$th error in the approximation of} $f\in\mathbb X$:
	$$\sigma_m[\mathcal X,\XX](f)=\sigma_m(f):=\inf\left\lbrace\left\Vert f-\sum_{j\in A}a_j\xx_j\right\Vert : a_j\in\mathbb F\,\forall j\in A, \vert A\vert\leq m\right\rbrace.$$
	\begin{definition}
		We say that a basis $\XB$ in a quasi-Banach space $\XX$ is \textit{greedy} if there is $\C>0$ such that for all $f\in\mathbb X$ and $m\in\mathbb N$,
		\begin{eqnarray}\label{defg}
			\Vert f-P_A(f)\Vert\leq \C\sigma_m(f),
		\end{eqnarray}
		whenever $A$ is a finite greedy set of $f$ of cardinality $m$. The smallest constant verifying \eqref{defg} is denoted by $\C_g$ and we say that $\XB$ is $\C_g$-greedy.
	\end{definition}
	
	In \cite{KT}, the authors characterized greediness in terms of unconditionality and democracy. To formalize the last notion, we need the following notation:
	\begin{align*}
		&\Ind_{\varepsilon, A}[\mathcal X,\mathbb X]=\Ind_{\varepsilon,A}:=\sum_{j\in A}\varepsilon_j \xx_j;&&\varepsilon=(\varepsilon_j)_{j\in A}\in\EE^{A};&&&\EE:=\{\lambda\in \FF\colon \left\vert \lambda\right\vert=1\}. 
	\end{align*}
	In general, $\varepsilon\in\EE^A$ is called a \textit{sign} and, for $f\in\XX$,
$$\varepsilon(f)\equiv\lbrace\sgn(\xx_n^*(f))\rbrace, n\in\mathbb N.$$	
	Thus, a basis $\XB$ is $\D$-\textit{democratic} with $\D>0$ if
	$$\left\Vert \Ind_A\right\Vert\leq \D\left\Vert\Ind_B\right\Vert,$$
		for any pair of finite sets with $\vert A\vert\leq\vert B\vert$, where $\Ind_A=\Ind_{\varepsilon,A}$ with $\varepsilon\equiv 1$. A similar characterization could be obtained by substituting superdemocracy for democracy, where $\XB$ is $\D_{s}$-\textit{superdemocratic} with $\D_{s}>0$ if
	$$\left\Vert \Ind_{\varepsilon, A}\right\Vert\leq \D_{s}\left\Vert\Ind_{\eta, B}\right\Vert,$$
for any pair of finite sets $A$ and $B$ with $\vert A\vert\leq\vert B\vert$, $\varepsilon\in \EE^{A}$ and $\eta\in\EE^{B}$.
	
	\begin{theorem}[{\cite{AABW2021,KT}}]\label{kt}
		A basis $\XB$ in a quasi-Banach space $\XX$ is greedy if and only if $\XB$ is unconditional and democratic or superdemocratic. 
	\end{theorem}
	
	In \cite{BB2017}, the authors introduced a new way to study greedy bases using polynomials of constant coefficients, where a polynomial of this type is basically an element of $\XX$ of the form $\alpha \Ind_{\varepsilon,A}=\sum_{j\in A}\alpha \varepsilon_j\xx_j$ where $\alpha\in\mathbb F$ and $\varepsilon\in\EE^A$.

	\begin{theorem}[\cite{BB2017}]
		Let $\XB$ be a basis in a Banach space $\XX$. The basis is greedy if and only if there is $\C>0$ such that for all $f\in\XX$,
		\begin{eqnarray}\label{op1}
			\Vert f-P_A(f)\Vert \leq \C\inf\lbrace\Vert f-\alpha\Ind_{\varepsilon, B}\Vert : \alpha\in\mathbb F, \vert B\vert=\vert A\vert, \varepsilon\in\EE^{B}\rbrace,
		\end{eqnarray}
		whenever $A$ is a finite greedy set of $f$.
	\end{theorem}
	
	The purpose of this paper is to relax the optimization problem given in \eqref{op1} trying to to fix the coefficient $\alpha$. That is, in the error
	$$\inf\lbrace\Vert f-\alpha\Ind_{\varepsilon, B}\Vert : \alpha\in\mathbb F, \vert B\vert=\vert A\vert, \varepsilon\in\EE^{B}\rbrace,$$
	we are taking the infimum over the sets, over the signs $\varepsilon$ and over $\alpha$. Here, we show that we can fix the value of $\alpha$  and take only the infimum over the sets since we can also fix the sign. For that, we introduce the following two new errors:
	
	\begin{itemize}
		\item For each $m\in \NN$ and each $f\in \XX$, let
		\begin{align*}
		\rho_m[\mathcal X,\mathbb X]=\rho_m(f):=\inf\left\lbrace \left\Vert f-\alpha\Ind_{\varepsilon, A }\right\Vert\colon \alpha=\min_{n\in A_m(f) }|\xx_n^*(f)|, \quad |A|=m, \varepsilon\in \EE^A\right\rbrace. 
		\end{align*}
		\item For each $m\in \NN$ and each $f\in \XX$, let
		\begin{align*}
			\varrho_m[\mathcal X,\mathbb X]=\varrho_m(f):=\inf\left\lbrace \left\Vert f-\alpha\Ind_{ A }\right\Vert\colon \alpha=\min_{n\in A_m(f) }|\xx_n^*(f)|, \quad |A|=m \right\rbrace. 
		\end{align*}
	\end{itemize}
	
	\begin{definition}
		We say that a basis $\XB$ is \textit{relaxed greedy for polynomials with constant coefficients} (RGPCC for short) if there is $\C>0$ such that for all $m\in\mathbb N$ and $f\in\XX$,
		\begin{eqnarray}\label{defgp}
			\Vert f-P_A(f)\Vert\leq \C\rho_m(f),\; \forall A\, \text{ greedy set of cardinality}\, m.
		\end{eqnarray}
		The smallest constant verifying \eqref{defgp} is denoted by $\C_{pg}$ and we say that $\XB$ is $\C_{pg}$-RGPCC.
	\end{definition}
	
	\begin{definition}
		We say that a basis $\XB$ is \textit{unsigned-relaxed greedy for polynomials with constant coefficients} (URGPCC for short) if there is $\C>0$ such that for all $m\in\mathbb N$ and $f\in\XX$,
		\begin{eqnarray}\label{defgp2}
			\Vert f-P_A(f)\Vert\leq \C\varrho_m(f),\; \forall A\, \text{ greedy set of cardinality}\, m.
		\end{eqnarray}
		The smallest constant verifying \eqref{defgp2} is denoted by $\C_{pgu}$ and we say that $\XB$ is $\C_{pgu}$-URGPCC.
	\end{definition}
	
	Here we will show the following result:
	\begin{theorem}\label{th1}
		Let $\XB$ be a basis in a quasi-Banach space $\XX$. The following are equivalent:
		\begin{itemize}
			\item[i)] $\XB$ is greedy.
			\item[ii)] $\XB$ is RGPCC.
			\item[iii)] $\XB$ is URGPCC.
		\end{itemize}
	Quantitatively, if $\XX$ is a $p$-Banach space with $0<p\leq 1$, then
	$$\C_{pgu}\leq \C_{pg}\leq \C_g.$$
	Moreover, if $\mathbb F=\mathbb R$, then
	\begin{eqnarray}\label{th1-1}
		\C_g\leq \min\lbrace A_p^2 \C_{pg}^2, \C_{pgu}^{p}(1+A_p^p\C_{pgu}^{2p}\min\lbrace B_p^p, A_p^p\C_{pgu}^{2p}\})^{1/p}\rbrace.
	\end{eqnarray}
		If $\mathbb F=\mathbb C$, there is a constant $\K=\K(\C_{pgu},p)$ such that
	\begin{eqnarray}\label{th1-2}
		\C_g\leq \min\lbrace A_p^2 \C_{pg}^2, \K(1+A_p^p\C_{pgu}^{2p}\min\lbrace B_p^p, A_p^p\K^{p}\})^{1/p}\rbrace.
	\end{eqnarray}
	\end{theorem}
	
	Additionally, we prove similar results for almost-greedy bases.
	\begin{definition}
		We say that a basis $\XB$ in a quasi-Banach space $\XX$ is \textit{almost-greedy} if there is $\C>0$ such that for all $f\in\mathbb X$,
		\begin{eqnarray}\label{defag}
			\Vert f-P_A(f)\Vert\leq \C\inf\lbrace\Vert f-P_B(f)\Vert : \vert B\vert\leq\vert A\vert\rbrace,
		\end{eqnarray}
		whenever $A$ is a finite greedy set of $f$. The least constant verifying \eqref{defag} is denoted by $\C_{ag}$ and we say that $\XB$ is $\C_{ag}$-almost-greedy.
	\end{definition}
	
	Related to these bases and following the idea about polynomials of constant coefficients, S. J. Dilworth and D. Khurana proved in \cite{DKhu2018} that a basis in a Banach space is almost-greedy if and only if there is $\C>0$ such that
	\begin{align}
		\|f-P_A(f)\|\le& \C\|f-a\Ind_{B}\|\label{toimprove1}
	\end{align}
	for all $f\in \XX$, $A$ a greedy set of $f$, $B\subset \NN$, $|B|\le |A|$, all $a\in \FF$ and $A< B$ or $B<A$, where $A<B$ means that $\max_{j\in A} j <\min_{i\in B} i$.\\
Our next theorem improves the above result.

	\begin{theorem}\label{thag}Let $\XB$ be a basis for a $p$-Banach space $\XX$. The following are equivalent:
		\begin{enumerate}[\rm i)]
			\item \label{ag} $\XB$ is almost greedy.
			\item\label{two} $\XB$ is quasi-greedy and superdemocratic.
		 \item \label{dis} There is $\C>0$ such that
			\begin{align*}
				\|f-P_A(f)\|\le& \C\|f-a\Ind_{\varepsilon, B}\|
			\end{align*}
			for all $f\in \XX$, $A$ a greedy set of $f$, $B\subset \NN$, $|B|\le |A|$ and $A\cap B=\emptyset$, $\varepsilon\in \EE^B$, and all $a\in \FF$.

			\item\label{end} There is $\C>0$ such that
			\begin{align*}
				\|f-P_A(f)\|\le& \C\|f-\min_{n\in A}\vert\xx_n^*(f)\vert \Ind_{\varepsilon,B}\|
			\end{align*}
			for all $f\in \XX$, $A$ a greedy set of $f$, $B\subset \NN$, $|B|\le |A|$ and $A< B$ or $B<A$ and all $\varepsilon\in \EE^{ B}$. 
			\item\label{end2} There is $\C>0$ such that
			\begin{align*}
				\|f-P_A(f)\|\le& \C\|f-\min_{n\in A}\vert\xx_n^*(f)\vert \Ind_{B}\|
			\end{align*}
			for all $f\in \XX$, $A$ a greedy set of $f$, $B\subset \NN$, $|B|= |A|$ and $A< B$ or $B<A$. 
		\end{enumerate}
	\end{theorem}
	
	In Theorem~\ref{thag}, the equivalence between \ref{ag} and \ref{two} is well known (see  \cite[Theorem 3.3]{DKKT2003}, \cite[Theorem 6.3]{AABW2021}), our contribution here are the other points, in particular \ref{dis} and \ref{end2}. 
	\section{Technical results}\label{tech}
	In this section, we focus our attention on some characterizations and relations between unconditionality and symmetry-like properties.
	
	As we have commented, in Theorem \ref{kt} it was proven that a basis in a Banach space is greedy if and only if it is democratic and unconditional and the behaviour of the constants is as follows: 
	$$\max\lbrace \K, \D\rbrace\leq \C_g\leq \K(1+\D).$$
	In the last decade, researchers have worked on improving this estimate since there are examples of $1$-democratic and $1$-unconditional bases that are not $1$-greedy. In 2006, F. Albiac and P.  Wojtaszczyk found necessary and sufficient conditions under which $\C_g=1$. Concretely, they introduced the notion of Property (A), extended and renamed in  \cite{DKOWS} by the so called symmetry for largest coefficients.
	
	\begin{definition}\label{defsym1}
		We say that a basis $\XB$ in a quasi-Banach space is \textit{symmetric for largest coefficients} (SLC for short) if there is $\C>0$ such that
		\begin{eqnarray}\label{defsym}
			\Vert f+\Ind_{\varepsilon, A}\Vert\leq \C\Vert f+\Ind_{\eta, B}\Vert,
		\end{eqnarray}
		for all $f\in\XX$ such that $\text{supp}(f)\cap (A\cup B)=\emptyset$, for all finite sets $A$ and $B$ such that $\vert A\vert\leq \vert B\vert$ and $A\cap B=\emptyset$ and all choices of signs $\varepsilon\in\EE^A$ and $\eta\in\EE^B$. The least constant verifying \eqref{defsym} is denoted by $\Delta$ and we say that $\XB$ is $\Delta$-SLC.
	\end{definition}
	
	Using this concept, in  \cite{DKOWS} we can find the following.
	\begin{theorem}
		A basis in a Banach space is greedy if and only if the basis is symmetric for largest coefficients and unconditional. Quantitatively, 
		$$\max\lbrace \Delta, \K\rbrace\leq \C_g\leq \Delta \K.$$
	\end{theorem}
	\begin{remark}\label{rem1}
		In the case of $p$-Banach spaces, in \cite{AABW2021} we find the same result with the following behaviour: 
		$$\max\lbrace \Delta, \K\rbrace\leq \C_g\leq A_p^2\Delta \K,$$
		where $A_p$ is defined as in \eqref{ap}. Also, changing the symmetry for largest coefficients by democracy for general $p$-Banach spaces, we have the bounds (see \cite{AABW2021})
		$$\C_g\leq \K(1+A_p^p\D^p\min\lbrace B_p^p,A_p^p\K^p\rbrace)^{1/p},$$
		where $B_p$ is defined as in \eqref{bp}.
	\end{remark}
	In this line, we modify the proof of \cite[Proposition 2.11]{AAB2023} to show that in order to prove that a basis is unconditional, it is enough consider approximations by vectors with nonnegative coefficients. This result will be used in the proof of Theorem \ref{th1}.

	 We use the following notation: for a given basis $\XB$,

	\begin{align*}
		\XX_{\RR}(\XB):=&\left\lbrace f\in \XX\colon \xx_n^*(f)\in \RR\quad\forall n\in \NN\right\rbrace;\\
		\XX_{+}(\XB):=&\left\lbrace f\in \XX\colon \xx_n^*(f)\in \RR_{\ge 0}\quad\forall n\in \NN\right\rbrace.
	\end{align*}

	\begin{lemma}\label{lemmacomplex+realuncond}Let $\XB$ be a basis for a $p$-Banach space $\XX$ over $\FF$ and $\C>0$. Suppose that for every $f \in \spn\{\XB\}$, $g\in \spn\{\XB\}\cap \XX_{+}(\XB)$ with $\supp(f)\cap \supp(g)=\emptyset$, 
		\begin{align}
			\|f\|\le \C\|f+g\|. \label{supuncond}
		\end{align}
		Then $\|f\|\le \K\|f+g\|$ for every $f, g\in\mathbb X$ with $\supp(f)\cap \supp(g)=\emptyset$, with $\K=\C^2$ if $\FF=\RR$ and $\K$ depending only on $\C$ and $p$ if $\FF=\CC$. Thus, $\XB$ is $\K$- unconditional. 
	\end{lemma}
	\begin{proof}
		By a standard density argument (or use Lemma~\ref{l:appr_fun_suppgeneral} below), it is enough to prove that \eqref{supuncond} holds with $\K$ instead of $\C$ for $f \in \spn\{\XB\}$, $g\in \spn\{\XB\}$ with $\supp(f)\cap \supp(g)=\emptyset$. If $\FF=\RR$,  choose disjointly supported \color{black} $f, g \in \spn\{\XB\}$, and write $g=g_1-g_2$ with $g_1,g_2\in \XX_{+}(\XB)\cap \spn\{\XB\}$, and $f, g_1, g_2$ pairwise disjointly supported. We have
		\begin{align*}
			\|f\|\le& \C\|f+g_1\|=\C\|-f-g_1\|\le \C^2\|-f-g_1-g_2\|=\C^2\|f+g\|, 
		\end{align*}
		and the proof is complete. \\
		Now suppose $\FF=\CC$, and let $\K_1:=\left(1+\C^{p}\right)^{\frac{1}{p}}$. By $p$-convexity and hypothesis
		\begin{align}
&\|g\|\le \left(\|g+f\|^p+\|f\|^p\right)^{\frac{1}{p}}\le \K_1\|g+f\|\nonumber\\
&\forall g\in \spn\{\XB\}\cap \XX_{+}(\XB), f\in \spn\{\XB\},\supp(f)\cap \supp(g)=\emptyset. \label{Ccuadrado}
		\end{align}
		We claim that there is $\K_2$ depending only on $p$ and $\C$ such that, for every $A\subset \NN$ with $0<|A|<\infty$, $f\in \spn\{\XB\}$ with $\supp(f)\cap A=\emptyset$, and $(\varepsilon_n)_{n\in A}\in \EE^A$, 
		\begin{align}
			\left\Vert \Ind_{\varepsilon, A}\right\Vert\le& \K_2\left\Vert \Ind_{\varepsilon, A}+f\right\Vert.\label{claim1}
		\end{align}
		To prove our claim, first choose a finite nonempty set $A\subset \NN$, $f\in \spn\{\XB\}$ with $\supp(f)\cap A=\emptyset$, and $(\varepsilon_n)_{n\in A}\in \EE^A$ so that
		\begin{align*}
			\left\vert \varepsilon_n-1\right\vert\le \left(2^{\frac{1}{p}}\K_1B_p\right)^{-1} \qquad \forall n\in A.
		\end{align*}
		Pick $B\subset A$ so that $\|\Ind_{D}\|\le \|\Ind_{B}\|$ for all $D\subset A$. By \eqref{Ccuadrado}  and Lemma~\ref{lemmapbanach}, 
		\begin{align*}
			\left\Vert \Ind_{B}\right\Vert^p\le& \K_1^p\left\Vert \Ind_{B}+\Ind_{\varepsilon, A\setminus B}+ f \right\Vert^p\\
			\le& \K_1^p\left\Vert \Ind_{\varepsilon, B}+\Ind_{\varepsilon, A\setminus B}+f\right\Vert^p+\K_1^p\left\Vert \Ind_{\varepsilon, B}-\Ind_{B}\right\Vert^p \\
			\le& \K_1^p\left\Vert \Ind_{\varepsilon, A}+f\right\Vert^p+  \K_1^p B_p^p \left(2^{\frac{1}{p}}\K_1B_p\right)^{-p}\left\Vert \Ind_{B}\right\Vert^p\\
			\le& \K_1^p\left\Vert \Ind_{\varepsilon, A}+f\right\Vert^p+\frac{1}{2}\left\Vert \Ind_{B}\right\Vert^p,
		\end{align*}
		so 
		\begin{align*}
			\left\Vert \Ind_{B}\right\Vert^p\le& 2\K_1^p\left\Vert \Ind_{\varepsilon, A}+f\right\Vert^p. 
		\end{align*}
		Another application of Lemma~\ref{lemmapbanach} gives 
		\begin{align*}
			\left\Vert \Ind_{\varepsilon, A}\right\Vert^p\le& \left\Vert \Ind_{A}\right\Vert^p+ \left\Vert \Ind_{\varepsilon, A}-\Ind_{A}\right\Vert^p\le \left\Vert \Ind_{A}\right\Vert^p+ B_p^p \left(2^{\frac{1}{p}}\K_1B_p\right)^{-p} \left\Vert \Ind_{B}\right\Vert^p
			\le\frac{3}{2}\left\Vert \Ind_{B}\right\Vert^p.
		\end{align*}
		Hence,
		\begin{align}
			\left\Vert \Ind_{\varepsilon, A}\right\Vert^p\le& 3\K_1^p\left\Vert \Ind_{\varepsilon, A}+f\right\Vert^p.\label{cercanosa1}
		\end{align}
		Now let $\{\nu_1,\dots,\nu_{j_1}\}\subset\EE$ be a set of minimum cardinality with the property that, for every $\varepsilon\in \EE$ there is $1\le j\le j_1$ such that $\left\vert \varepsilon-\nu_j\right\vert \le  \left(2^{\frac{1}{p}}\K_1B_p\right)^{-1}$, and fix $A\subset \NN$ a finite nonempty set, $f\in \spn\{\XB\}$ with $\supp(f)\cap A=\emptyset$, and $(\varepsilon_n)_{n\in A}\in \EE^A$. Choose $I\subset \{1,\dots,j_1\}$ and $(A_j)_{j\in I}$ a partition of $A$ such that $\left\vert \varepsilon_n-\nu_j\right\vert \le  \left(2^{\frac{1}{p}}\K_1B_p\right)^{-1}$ for each $n\in A_j$ and each $j\in I$. By \eqref{cercanosa1}, 
		\begin{align*}
			\left\Vert \Ind_{\varepsilon, A}\right\Vert^p\le&\sum_{j\in I}\left\Vert \nu_j^{-1}\Ind_{\varepsilon, A_j}\right\Vert^p\le \sum_{j\in I}3\K_1^p\left\Vert \nu_j^{-1}\Ind_{\varepsilon, A_j}+\nu_j^{-1}\Ind_{\varepsilon, A\setminus A_j}+\nu_j^{-1}f \right\Vert^p\\
			=& 3|I|\K_1^p\left\Vert \Ind_{\varepsilon, A}+f\right\Vert^p\le 3j_1\K_1^p\left\Vert \Ind_{\varepsilon, A}+f\right\Vert^p.
		\end{align*}
		Note that $\K_1$ only depends on $\C$ and $p$, whereas $j_1$ only depends on $\K_1$ and $p$. Thus, we have proven \eqref{claim1} with $\K_2:= \left(3j_1\right)^{\frac{1}{p}}\K_1$. To complete the proof of the Lemma, fix $f,g\in \spn\{\XB\}\setminus \{0\}$ with disjoint support. Let $A:=\supp(g)$, and define $\YB=(\yy_n)_{n\in \NN}$ by
		\begin{align*}
			\yy_n=&\begin{cases}
				\left\vert \xx_n^{*}(g)\right\vert\xx_n & \text{if } n\in A;\\
				\xx_n & \text{if } n\not\in A. 
			\end{cases}
		\end{align*}
		Then $\YB$ is a basis for $\XX$ and, since each $\yy_n$ is the product of $\xx_n$ by a positive scalar, \eqref{supuncond} also holds substituting $\YB$ and $\XX_{+}(\YB)$ for $\XB$ and $\XX_{+}(\XB)$ respectively. \color{black} Since $\K_2$ depends only on $\C$ and $p$, the same argument given above shows that \eqref{claim1} holds for $\YB$ as well. Therefore, 
		\begin{align*}
			\|g\|=&\left\Vert \Ind_{\varepsilon(g), A}[\YB,\XX]\right\Vert \le \K_2\left\Vert \Ind_{\varepsilon(g), A}[\YB,\XX]+f\right\Vert=\K_2\|g+f\|.
		\end{align*}
		This completes the proof of the Lemma, with $\K=\K_2$. 
	\end{proof}

	\begin{remark}\label{remarknotbounded}\rm While we have defined our basis to be bounded with bounded dual basis, neither condition is required in the proof of Lemma~\ref{lemmacomplex+realuncond}. 	
	\end{remark}
	
	With respect to the symmetry for largest coefficients, in \cite{BDKOW2019}, the authors proved that in the case of Schauder bases in Banach spaces, in the definition of the SLC property it is enough to consider elements $f\in\mathbb X$ with finite support (their proofs hold for general Markushevich bases).
	
	The same result holds in the context of general bases of $p$-Banach spaces. First of all, we consider the extension for Schauder bases, where the proof is immediate.
	
	\begin{lemma}\label{l:appr_fun_suppSchauder}
		Let $\XB$ be a Schauder basis of a $p$-Banach space $\mathbb X$. 				Suppose $D$ is a finite subset of $\mathbb N$, and $f \in \mathbb X$
		satisfies $\supp(f) \cap D = \emptyset$.
		Then, for any $\varepsilon > 0$ there exists a finitely supported $y \in \mathbb X$ such that
		$\Vert f - y\Vert < \varepsilon$, $\supp(y) \cap D = \emptyset$,
		and $\max_j |\xx_j^*(f)| = \max_j |\xx_j^*(y)|$.
	\end{lemma}
	\begin{proof}
		Choose $n_0\in \mathbb N$ so that $|\xx_{n_0}^*(f)|=\|f\|_{\ell_{\infty}}$, and $n_1\ge n_0$ so that $\|f-S_{n_1}(f)\|<\varepsilon$, and let $y:=S_{n_1}(f)$. 
	\end{proof}
	
	For general bases, the result can be obtained by adapting the proofs of \cite[Lemma 7.1]{BB2020} or \cite[Lemma 2.2]{Oikhberg2017}. Nevertheless, we give a proof for the sake of completion. 
	
	\begin{lemma}\label{l:appr_fun_suppgeneral}
		Let $\XB$ be a basis of a $p$-Banach space $\XX$ and let $f \in \mathbb X$. The following hold: 
		\begin{enumerate}[\rm i)]
			\item \label{notempty2}If $\supp(f)\not=\emptyset$,  for every $D\subset \NN$ finite and every $\varepsilon > 0$ there exists $g\in \spn\lbrace{\XB\rbrace}$ such  that 
\begin{align}
&\Vert f - g\Vert < \varepsilon, \qquad P_D(g)=P_D(f), \qquad \max_{j\in D^c} |\xx_{j}^*(f)| = \max_{j\in D^c} |\xx_{j}^*(g)|\label{condempty23}.
\end{align}			
In particular, if $D\cap \supp(f)=\emptyset$, then $D\cap \supp(g)=\emptyset$ and $\max_{j} |\xx_j^*(g)|=\max_{j} |\xx_j^*(f)|$. 
\item \label{empty2}If $\supp(f)=\emptyset$, for every $D\subset \NN$ finite and every $\varepsilon > 0$ there exists $g\in \spn\lbrace{\XB\rbrace}$ such  that
			$\Vert f - g\Vert < \varepsilon$, $\max_j |\xx_j^*(g)|<\epsilon$, and $P_D(g)=P_D(f)=0$. 
		\end{enumerate}		
					
	\end{lemma}
	\begin{proof}
	Pick $n_0\not\in D$ and set 
$$
c:=\max_{n\in D\cup \{n_0\}}\{\|\xx_n^*\|+\|\xx_n\|\}. 
$$	
To prove \eqref{condempty23}, first we consider the  case $D\cap \supp(f)=\emptyset$. We may assume that $\max_j |\xx_j^*(f)| = 1$ (then $\Vert f \Vert \geq 1/c$) and $\varepsilon < 1/(2c)$. 	
 Set $\delta>0$ so that 		\begin{align*}
			&0<\delta\le \frac{\varepsilon}{3c^2\|f\|};&&\left(\delta^p + \left(\frac{c \delta}{1 - c \delta}\right)^p( \Vert f\Vert^p + \delta^p)  \right)^{\frac{1}{p}} < \varepsilon.
		\end{align*}
		As $\spn\lbrace{\XB\rbrace}$ is dense in $\mathbb X$, there exists $h \in\spn\lbrace{\XB\rbrace}$ such that $\Vert f - h\Vert < \delta/\Vert P_D^c\Vert$. Let $u = P_D^c(h)$, then
		$\Vert f - u\Vert = \Vert P_D^c(f-h)\Vert < \delta$. For every $j$, $|\xx_j^*(f-u)| < c \delta$,
		hence $C = \max_j |\xx_j^*(u)| \in (1 - c \delta, 1 + c \delta)$.
		Now let $g = u/C$. Then $\max_j |\xx_j^*(g)| = 1$, and
		\begin{align*}
			\Vert f-g\Vert \leq& \left(\Vert f - u\Vert^p + |1 - C^{-1}|^p \Vert u\Vert^p \right)^{\frac{1}{p}}<
			\left(\delta^p + \left(\frac{c \delta}{1 - c\delta}\right)^p( \Vert f\Vert^p + \delta^p) \right)^{\frac{1}{p}} < \varepsilon .
		\end{align*}
This completes the proof of \eqref{condempty23} when $P_D(f)=0$. \\
Now suppose $P_D(f)\not=0$, and let $f_1:=f-P_D(f)$. By the previous case,  there is $g_1\in \spn\lbrace{\XB\rbrace}$ so that 	 \eqref{condempty23} holds for $f_1$, $g_1$, $D$ and $\epsilon$. Now let $g:=g_1+P_D(f)$; it is clear that \eqref{condempty23} holds. \\		
Let us now prove \ref{empty2}: If $\BB$ is a Markushevich basis, then $f=0$ so we can take $g=0$. Otherwise, choose $\delta>0$ so that
		\begin{align*}
			\delta\left(1+c^p\right)^{\frac{1}{p}}<\epsilon,
		\end{align*}
and let $f_0:=f+\delta \xx_{n_0}$. By \ref{notempty2}, there is $g\in \spn\lbrace{\XB\rbrace}$ such that $\|g-f_0\|\le \delta$, $ \max_j |\xx_j^*(g)| =\delta$ and $P_D(g)=P_D(f)=0$. We have 
		\begin{align*}
			\|f-g\|\le& \left(\|g-f_0\|^p+\|f_0-f\|^p\right)^{\frac{1}{p}}\le \left(\delta^p+c^p\delta^p\right)^{\frac{1}{p}}=\delta\left(1+c^p\right)^{\frac{1}{p}}<\epsilon. 
		\end{align*}
	\end{proof}

	\begin{remark}\rm \label{remarkunbounded1} While our bases are by definition bounded with bounded dual bases, note that the proof of Lemma~\ref{l:appr_fun_suppgeneral} holds without either condition. \end{remark}

	\section{Proof of Theorem \ref{th1}}\label{proof1}

	\begin{lemma}\label{lemmagreedycar2}
		Let $\XB$ be a basis in a $p$-Banach space $\XX$. Suppose that $\XB$ is $\C_{pg}$-RGPCC. Then $\XB$ is $\K$-unconditional and $\Delta$-symmetric for largest coefficients with 
		$$\max\lbrace \K, \Delta\rbrace\leq \C_{pg}.$$
	\end{lemma}
	\begin{proof}
		Fix $f,g\in \spn\{\XB\}$ such that $\supp(f)\cap \supp(g)=\emptyset$. Let $A:=\supp(g)$, $B:=\supp(f)$ and $t_0:=1+\max_{n\in \NN}\left(|\xx_n^*(g)|+|\xx_n^*(f)|\right)$. Pick $n_0\in \NN\setminus (A\cup B)$, and let $\varepsilon_n=\sgn(\xx_n^*(g))$ for all $n\in A$ and $\varepsilon_{n_0}=1$. Note that $A_0$ is the only greedy set of $h:=f+g+t_0\Ind_{\varepsilon, A_0}$ of cardinality $|A_0|=|A|+1$, and that $\min_{n\in A_0}\left\vert \xx_n^*(h)\right\vert=t_0$. Thus, 
		\begin{eqnarray*}
			\|f\|=&\|h-P_{A_0}(h)\|\le& \C_{pg}\,\rho_{|A|+1}\left(h\right)\le \C_{pg}\|h-t_0\Ind_{\varepsilon, A_0}\|=\C_{pg}\|f+g\|, 
		\end{eqnarray*}
		which proves that $\XB$ is $\K$-unconditional with $\K\leq \C_{pg}$. 
		
		Now fix $f\in \XX$ with finite support, and $A, B\subset \NN$ nonempty, finite, disjoint and disjoint from $\supp(f)$ with $|A|=|B|$, $\varepsilon\in \EE^A$ and $\eta\in \EE^B$.  Given $0<\epsilon<1$, we have 
		\begin{align*}
			\|f+\Ind_{\varepsilon, A}\|^p=&\|f+\Ind_{\varepsilon, A}+(1+\epsilon)\Ind_{\eta,B}-(1+\epsilon)\Ind_{\eta,B}\|^p\\
			\le& \C_{pg}^p\rho_{|B|}^p\left(f+\Ind_{\varepsilon, A}+(1+\epsilon)\Ind_{\eta,B}\right)\\
			\le&\C_{pg}^p\|f+\Ind_{\varepsilon, A}+(1+\epsilon)\Ind_{\eta,B}-(1+\epsilon)\Ind_{\varepsilon,A}\|^p\\
			\le&\C_{pg}^p\|f+\Ind_{\eta, B}\|^p+\epsilon^p\|\Ind_{\varepsilon, A}\|^p+\epsilon^p\|\Ind_{\eta, B}\|^p.
		\end{align*}
		Letting $\epsilon$ tend to zero we obtain
		\begin{align*}
			\|f+\Ind_{\varepsilon, A}\|^p\le& \C_{pg}^p\|f+\Ind_{\eta,B}\|^p.
		\end{align*}
		Hence, by Lemma \ref{l:appr_fun_suppgeneral}, $\XB$ is $\Delta$-SLC with $\Delta\leq \C_{pg}$.
	\end{proof}
	
	Our next task is obtain greediness under hypotheses that are formally weaker than those of  Lemma~\ref{lemmagreedycar2}, to complete our characterization of greediness.\color{black}

	\begin{lemma}\label{lemmagreedy3}
		Let $\XB$ be a basis for a $p$-Banach space $\XX$. Suppose that $\XB$ is $C_{pgu}$-URGPCC. Then $\XB$ is  unconditional and democratic. Quantitatively, if $\mathbf F=\mathbb R$, then
			$$\max\lbrace \K,\D\rbrace\leq \C_{pgu}^2.$$
			If $\mathbb F=\mathbb C$, there is $\mathbf K=\mathbf K(C_{pgu},p)$ such that
		$$ \K\leq \mathbf \K,\; \D\leq \C_{pgu}^2.$$
	\end{lemma}
	\begin{proof}
		First, we prove that it that $\XB$ is  unconditional. Fix $f \in \spn\{\XB\}$, $g\in \spn\{\XB\}\cap \XX_{+}$ so that $\supp(f)\cap \supp(g)=\emptyset$. Set  $t_0:=1+\max_{n\in \NN}\left(\xx_n^*(g)+|\xx_n^*(f)|\right)$, and pick $n_0>\supp(f)\cup \supp(g)$. Note that $A:=\supp(g)\cup \{n_0\}$ is the only greedy set of $h:=f+g+t_0\Ind_{A}$ of cardinality $|A|$, and that $\min_{n\in A}\left\vert \xx_n^*(h)\right\vert=t_0$. Thus, 
		\begin{eqnarray}
			\|f\|=&\|h-P_{A}(h)\|\le \C_{pgu}\varrho_{|A|}\left(h\right)\le \C_{pgu} \|h-t_0\Ind_{A}\|=&\C_{pgu}\|f+g\|. \label{supunc+}
		\end{eqnarray}
		By Lemma~\ref{lemmacomplex+realuncond}, $\XB$ is $\K$-unconditional with $\K=\C_{pgu}^2$ if $\mathbb F=\mathbb R$ and $\K\leq \mathbf \K(\C_{pgu},p)$ if $\mathbb F=\mathbb C$. To prove that it is democratic, pick $A, B\subset \NN$ with $0<|A|\le |B|<\infty$ and $A\not=B$, and choose $B_0\subset B$ so that $|B_0|=|A|$. Given $0<\epsilon<1$, we have 
		\begin{align*}
			\|\Ind_{A}\|^p=&\|\Ind_{A\cap B_0}+\Ind_{A\setminus B_0}+(1+\epsilon)\Ind_{B_0\setminus A}-(1+\epsilon)\Ind_{B_0\setminus A}\|^p\\
			\le& \C_{pgu}^p \varrho^p_{|B_0\setminus A|}\left(\Ind_{A\cap B_0}+\Ind_{A\setminus B_0}+(1+\epsilon)\Ind_{B_0\setminus A}\right)\\
			\le& \C_{pgu}^p\|\Ind_{A\cap B_0}+\Ind_{A\setminus B_0}+(1+\epsilon)\Ind_{B_0\setminus A}-(1+\epsilon)\Ind_{A\setminus B_0}\|^p\\
			\le& \C_{pgu}^p\|\Ind_{B_0}\|^p+\C_{pgu}^p\epsilon^p\|\Ind_{B_0\setminus A}\|^p+\C_{pgu}^p\epsilon^p\|\Ind_{A\setminus B_0}\|^p. 
		\end{align*}
		Letting $\epsilon$ tend to zero we obtain $\|\Ind_{A}\|\le \C_{pgu}\|\Ind_{B_0}\|$. If $B_0=B$, there is nothing else to prove. Otherwise, an application of \eqref{supunc+} with $f=\Ind_{B_0}$ and $g=\Ind_{B\setminus B_0}$ gives $\|\Ind_{B_0}\|\le \C_{pgu}\|\Ind_{B}\|$. It follows that 
		\begin{align*}
			\|\Ind_{A}\|\le & \C_{pgu}^2\|\Ind_{B}\|,
		\end{align*}
		so the basis is $D$-democratic with $\D\leq \C_{pgu}^2$.
	\end{proof}
	
	\begin{proof}[Proof of Theorem \ref{th1}]
		First, we show i) $\Leftrightarrow$ ii). For that, of course, if $\mathcal X$ is $\C_g$-greedy, then it is trivial that $\XB$ is $\C_{pg}$-RGPCC with $\C_{pg}\leq \C_g$. Now, if $\XB$ is $\C_{pg}$-RGPCC, then, using Lemma \ref{lemmagreedycar2}, the basis is $\K$-unconditional and $\Delta$-SLC with $\max\lbrace \K, \Delta\rbrace\leq \C_{pg}$. Hence, using now Remark \ref{rem1}, the basis is $\C_g$-greedy with $\C_g\leq A_p^2\C_{pg}^2$.
		Now, we show that i) $\Leftrightarrow$ iii). Again, if $\mathcal X$ is $C_g$-greedy, then it is trivial that $\XB$ is $\C_{pgu}$-RGPCC(II) with $\C_{pgu}\leq \C_g$. Now, using Lemma \ref{lemmagreedy3}, the basis is $\D$-democratic and $\K$-unconditional with $\D\leq \C_{pgu}^2$ and $\K$ depends on $\C_{pgu}$ and $p$. Then, using now Remark \ref{rem1}, the basis is greedy and the constant is as it appears in \eqref{th1-1} and \eqref{th1-2}.
	\end{proof}
	
	\section{Proof of Theorem \ref{thag}}\label{proof2}
	In this section, we prove Theorem \ref{thag}. In addition to the classical characterization of almost greedy bases as those that are quasi-greedy and democratic or superdemocratic ( \cite[Theorem 3.3]{DKKT2003}, \cite[Theorem 6.3]{AABW2021}), we will use the following result from the literature (see \cite[Lemma 2.2]{DKKT2003}, \cite[Theorem 4.3]{AABW2021})

\begin{theorem}\cite[Theorem 4.13]{AABW2021}\label{theoremflattening}
If $\XB$ is a $\C_q$-quasi-greedy basis in a $p$-Banach space with $0<p\leq 1$, then  for all $f\in \XX$ and $A$ a greedy set of $f$,
	\begin{eqnarray}\label{tc}
\min_{n\in A} \left\Vert \Ind_{\varepsilon(f), A}\right\Vert\le \C\|f\|. 
		\end{eqnarray}
		The minimum constant $\C$ for which the above inequality holds is denoted $\Gamma_t$, and it's no greater than $\C_q^2\eta_p(\C_q)$, where 
\begin{align*}
\eta_p(u)=&\min_{0<t<1}(1-t^p)^{-1/p}(1-(1+A_p^{-1}u^{-1}t)^{-p})^{-1/p} \;\qquad\forall u>0. 
\end{align*}		
\end{theorem}
Bases for which \eqref{tc} holds have been recently named `truncation quasi-greedy' (see \cite{AABBL2023}), but have been known since the early days of the theory (see \cite{DKKT2003}) (it is known that there are bases satisfying \eqref{tc} that are not quasi-greedy, see for example \cite{AABBL2023} or \cite[Proposition 4.8]{DKKT2003}).

	\begin{proof}[Proof of Theorem \ref{thag}]
and		\ref{ag}$\Longleftrightarrow$ \ref{two} follows from \cite[Theorem 3.3]{DKKT2003}, \cite[Theorem 6.3]{AABW2021} , whereas \ref{dis} $\Longrightarrow$ \ref{end} $\Longrightarrow$ \ref{end2} is immediate.		\\

		\ref{two}$\Longrightarrow$ \ref{dis} Let $\D_s$, $\C_{q}$ and $\Gamma_{t}$ be the superdemocracy, quasi-greedy and truncation quasi-greedy constants of $\XB$ respectively. Given $f$, $A$, $B$, and $\varepsilon$ as in the statement, set $b:=\min_{n\in A}|\xx_n^*(f)|$. We may assume $b>0$ and we consider two cases:\\
		
		Case 1: If $|a|\le 2b$, let $D\subset \NN$ be a greedy set of $g:=f-a\Ind_{\varepsilon, B}$ of minimum cardinality that contains $A$. 
		Then 
		\begin{align*}
		\min_{n\in D}|\xx_n^*(g)|=&b, 
		\end{align*}
		so, 
		\begin{align}
			\|f-P_A(f)\|^p\le& \C_q^p\|f\|^p\le \C_q^p\|g\|^p+\C_q^p2^p\left(2^{-1}|a|\|\Ind_{\varepsilon, B}\|\right)^p\nonumber\\
	\le&  \C_q^p\|g\|^p+\C_q^p\D_s^p2^p\left(2^{-1}|a|\|\Ind_{\varepsilon(g), D}\|\right)^p\nonumber\\
	\le&\C_q^p(1+2^p\D_s^p\Gamma_t^p)\|g\|^p. \label{achico}
		\end{align}
		
		Case 2: If $|a|>2b$, set $g$ as before, and let $D$ be a greedy set of $g$ of minimum cardinality containing $B$. Then 
		\begin{align*}
		d:=\min_{n\in D}|\xx_n^*(g)|=&\min_{n\in B}|\xx_n^*(g)|\ge \frac{|a|}{2}. 
		\end{align*}
so
\begin{align*}
\|a\Ind_{\varepsilon, B}\|\le& 2 \D_s d\|\Ind_{\varepsilon(g), D}\|\le 2\D_s \Gamma_t\|g\|. 
\end{align*}
Hence, 
\begin{align}
\|f-P_A(f)\|^p\le& \C_q^p\|f\|^p\le \C_q^p\|g\|^p+\C_q^p\left(\|a\Ind_{\varepsilon, B}\|\right)^p\nonumber\\
\le&\C_q^p(1+2^p\D_s^p\Gamma_t^p)\|g\|^p. \label{agrande}
\end{align}
Combining \eqref{achico} and \eqref{agrande} we obtain  \ref{dis}, with 
		\begin{align*}
		\C\le& \C_q^p(1+2^p\D_s^p\Gamma_t^p).
		\end{align*}

\ref{end2}$\Longrightarrow$ \ref{two} 

Choose a finite set $A$, $\varepsilon\in \EE^A$, and $f\in \XX$ so that 
\begin{align*}
\left\vert \left\lbrace n\in \supp(x): \left\vert \xx_n^*(f)\right\vert \ge 1\right\rbrace\right\vert \ge |A|. 
\end{align*}
We claim that 
\begin{align}
\|\Ind_{\varepsilon, A}\|\le B_p \C^2 2^{\frac{1}{p}-1}\|f\|. \label{claim3}
\end{align}
To prove the claim, first assume that $f\in \spn\lbrace{\XB\rbrace}$ and  apply Lemma~\ref{lemmapbanach} to find $A_0\subset A$ such that 
\begin{align*}
\|\Ind_{\varepsilon, A}\|\le B_p\|\Ind_{A_0}\|. 
\end{align*}
Pick $B$ a greedy set of $f$ with $|B|=|A_0|$, let  $f_1:=\frac{f}{\min_{n\in B}\left\vert \xx_n^*(f)\right\vert}$, and $g:=f_1-P_B(f_1)$, and choose $D>\supp(f)\cup A$ so that $|D|=|A_0|$. By hypothesis and $p$-convexity, 
\begin{align*}
\|\Ind_{A_0}\|^p=&\|\Ind_{A_0}+\Ind_{D}-\Ind_{D}\|^p\le \C^p \|\Ind_{D}\|^p\le 2^{-p}\C^p\left\Vert \Ind_{D}+g\right\Vert^{p}+2^{-p}\C^p\left\Vert \Ind_{D}-g\right\Vert^{p}.
\end{align*}
On the other hand, as $B$ is a greedy set of both $\Ind_{D}+f_1$ and $\Ind_{D}-f_1$, we have 
\begin{align*}
\left\Vert \Ind_{D}+g\right\Vert=&\left\Vert \Ind_{D}+f_1-P_B(\Ind_{D}+f_1)\right\Vert\le \C\|\Ind_{D}+f_1-\Ind_{D}\|\le \C\|f\|;\\
\left\Vert \Ind_{D}-g\right\Vert=&\left\Vert \Ind_{D}-f_1-P_B(\Ind_{D}-f_1)\right\Vert\le \C\|\Ind_{D}-f_1-\Ind_{D}\|\le \C\|f\|.
\end{align*}
Combining the above inequalities we obtain \eqref{claim3} for $f\in \spn\lbrace{\XB\rbrace}$, and the general case follows at once by Lemma~\ref{l:appr_fun_suppgeneral}. In particular, it follows that $\XB$ is superdemocratic. To prove that it is also quasi-greedy, fix $f\in \XX$ and $A$ a greedy set of $f$,  and pick $B\subset \NN$ with $B>A$, and $|A|=|B|$. By hypothesis and \eqref{claim3} we have
\begin{align*}
\|f-P_A(f)\|^p\le& \C^p \|f-\min_{n\in A}|\xx_n^*(f)|\Ind_{B}\|^p\le \C^p\|f\|+\C^p\min_{n\in A}|\xx_n^*(f)|^p\|\Ind_{B}\|^p\\
\le& \C^p\|f\|^p+B_p^p\C^{3p}2^{1-p}\|f\|^p,
\end{align*}
and the proof is complete. 
	\end{proof}
	
\bigskip
\noindent \textbf{Funding :} D. González was partially supported by  ESI International Chair@ CEU-UCH. The last two authors were partially supported by the Grant PID2022-142202NB-I00 (Agencia Estatal de Investigación, Spain). Miguel Berasategui was partially supported by the grants CONICET PIP
11220200101609CO y ANPCyT PICT 2018-04104 (Consejo Nacional de Investigaciones Científicas y Técnicas y Agencia Nacional de Promoción de la Investigación, el Desarrollo Tecnológico y la Innovación, Argentina) .


\begin{thebibliography}{99}
		%
		%
		%
		%
		%
		\bibitem{AAB2023} \textsc{F. Albiac, J.L. Ansorena, M. Berasategui},  \emph{Elton's unconditionality as a threshold-free form of greediness}, J. Funct. Anal 285 (7), (2023). 
%
			
		%
		\bibitem{AABBL2023}  \textsc{F. Albiac, J.L. Ansorena, M. Berasategui, P.M. Berná, S. Lassalle}, 
		\emph{Bidemocratic bases and their connections with other greedy-type bases}, Constr. Approx. \textbf{57} 125-160, (2023). 
		
		%
		\bibitem{AABW2021} \textsc{F. Albiac, J.L. Ansorena, P.M. Berná, P. Wojtaszczyk}, 
		\emph{Greedy approximation for biorthogonal systems in quasi-Banach spaces}, 
		Dissertationes Math. \textbf{560} (2021), 1--88. 
		%
		%
		
		\bibitem{AAW2019} \textsc{F. Albiac, J. L. Ansorena, P. Wojtaszczyk}, 
		\emph{Conditional Quasi-Greedy Bases in Non-superreflexive Banach Spaces}, 
		Constr. Approx. \textbf{49} (2019), 103--122. 
		%
		\bibitem{AAW2020} \textsc{F. Albiac, J. L. Ansorena, P. Wojtaszczyk}, 
		\emph{Quasi-greedy bases in  $\ell_p$ ($0<p<1$) are democratic}, J. Funct. Anal. 
		\textbf{280} (7) (2020), 108871.
		
		\bibitem{Aoki} \textsc{T. Aoki},
		\emph{On a certain class of linear metric spaces}, Proc. Imp. Acad. Tokyo \textbf{18} (1942), 588--594.
		
		%
		%
		\bibitem{BB2020} \textsc{M. Berasategui, P. M. Bern\'a}, 
		\emph{Quasi-greedy bases for sequences with gaps}, 
		Nonlinear Anal. Theory Methods Appl, \textbf{208} (2021), 112294. 
		%
		%
		\bibitem{B2019} \textsc{P. Bern\'{a}}, 
		\emph{Equivalence between almost greedy and semi-greedy bases}, 
		J. Math. Anal. Appl. \textbf{470}  (2019), 218--225.
		%
		\bibitem{B2020} \textsc{P. Bern\'{a}}, 
		\emph{Characterization of Weight-semi-greedy bases}, 
		J. Fourier Anal. Appl. (2020) \textbf{26} (1) (2020), 1--21. 
		%
		\bibitem{BB2017} \textsc{P. Bern\'{a}, \'{O}. Blasco}, 
		\emph{Characterization of greedy bases in Banach spaces}, 
		J. Approx. Theory. \textbf{215} (2017), 28--39. 
		
		%
		%
		\bibitem{BDKOW2019} \textsc{P. Bern\'{a}, S. J. Dilworth, D. Kutzarova, T. Oikhberg, B. Wallis}, 
		\emph{The weighted Property (A) and the greedy algorithm}, 
		J. Approx. Theory, \textbf{248} (2019), 105300. 
		
		        
		%
		%
		
		\bibitem{DKhu2018} S. J. Dilworth, D. Khurana. \emph{Characterizations of almost greedy and partially greedy bases}. Jaen J. Approx. 11 (2019), 115--137.
		
		\bibitem{DKK2003} \textsc{S. J. Dilworth, N. J. Kalton, D. Kutzarova}, 
		\emph{On the existence of almost greedy bases in Banach spaces},
		Studia Math.  \textbf{159} (2003), 67--101.
		%
		\bibitem{DKKT2003} \textsc{S. J. Dilworth, N. J. Kalton, D. Kutzarova,  V. N. Temlyakov}, 
		\emph{The thresholding  greedy  algorithm, greedy  bases,  and  duality}, 
		Constr. Approx. \textbf{19} (2003), 575--597.
		%
		%
		\bibitem{DKOWS} \textsc{S. J. Dilworth,  D. Kutzarova, E. Odell, T. Schlumprecht, A. Zsák}, \emph{Renorming spaces with greedy bases}, J. Approx. Theory  \textbf{188}  (2014), 39--56.
		%
		%
		%
		%
		%
		\bibitem{KT} \textsc{S. V. Konyagin, V. N. Temlyakov}, 
		\emph{A remark on greedy approximation in Banach spaces}, 
		East J. Approx. \textbf{3} (1999), 365--379.
		%
		%
		\bibitem{Oikhberg2017} \textsc{T. Oikhberg}, 
		\emph{Greedy algorithms with gaps}, 
		J. Approx. Theory, 255 (2018), 176--190.
		
			\bibitem{Rolewicz} \textsc{S. Rolewicz},
		\emph{Locally bounded linear topological spaces}, English, with Russian summary, Bull. Acad. Polon. Sci. Cl. III. \textbf{5} (1957), 471--473, XL.
		
		%
		%
		%
		%
		%
		\bibitem{W2000} 
		\textsc{P. Wojtaszczyk}, 
		\emph{Greedy algorithm for general biorthogonal systems}, 
		J.  Approx. Theory \textbf{107} (2000), 293--314.
	\end{thebibliography}
\end{document}